\documentclass[12pt]{amsart}
\usepackage{amsmath,amssymb,amsthm}
\theoremstyle{plain}
\newtheorem{thm}{Theorem}[section]
\newtheorem{lem}[thm]{Lemma}
\newtheorem{cor}[thm]{Corollary}
\newtheorem{prop}[thm]{Proposition}
\newtheorem{defn}[thm]{Definition}

\newtheorem{example}[thm]{Example}

\newtheorem{rem}[thm]{Remark}
\newtheorem{obs}[thm]{Observation}
\newcommand{\simrightarrow}{\mathrel{\stackrel{\sim}{\rightarrow}}}
\newcommand{\inv}{{}^{-1}}

\newcommand{\id}{\text{id}}
\newcommand{\R}{\mathbb{R}}
\makeatletter
    
    \@addtoreset{equation}{section}
  \makeatother
\usepackage{hyperref}

\makeatother
\begin{document}

\title[Relation between spherical designs through a Hopf map]
{Relation between spherical designs through a Hopf map}
\author{Takayuki Okuda}
\address{Department of Mathematics, Graduate School of Science, Hiroshima University 1-3-1 Kagamiyama, Higashi-Hiroshima, 739-8526 JAPAN}
\email{okudatak@hiroshima-u.ac.jp}
\date{}
\maketitle

\begin{abstract}
Cohn--Conway--Elkies--Kumar [Experiment.~Math.~(2007)]
described that 
one can construct a family of designs on $S^{2n-1}$ from a design on $\mathbb{CP}^{n-1}$.
In this paper, we prove their claim for the case where $n=2$.
That is, 
we give an algorithm to construct $2t$-designs on $S^{3}$ 
as products through a Hopf map $S^3 \rightarrow S^2$ of a $t$-design on $S^2$ and a $2t$-design on $S^1$.
\end{abstract}

\section{Introduction}\label{subsection:intro:sphere}

The purpose of this paper is to give an algorithm to make 
a spherical $2t$-design $X$ on $S^3$ with $|X| = (2t+1) |Y|$
from a given spherical $t$-design $Y$ on $S^2$.

We write $S^d$ for the unit sphere in the $(d+1)$-dimensional Euclidean space $\mathbb{R}^{d+1}$.
The concept of spherical designs on $S^d$ were introduced by 
Delsarte--Goethals--Seidel \cite{Delsarte-Goethals-Seidel77spherical} in 1977
as follows:
For a fixed $t \in \mathbb{N}$,
a finite subset $X$ of $S^d$ is called a (spherical) $t$-design on $S^d$
if 
\begin{equation}
\frac{1}{|X|}\sum_{x \in X} f(x) = \frac{1}{|S^{d}|} \int_{S^d} f d\mu_{S^d} \label{intro:spherical_design}
\end{equation}
for any polynomial $f$ of degree at most $t$.
Note that 
the left hand side and the right hand side in \eqref{intro:spherical_design}
are the averaging values of $f$ on $X$ and that on $S^{d}$,
respectively.
(see Definition \ref{defn:spherical_design} for more details).
The development of spherical designs until 2009 can be found in Bannai--Bannai \cite{Bannai-Bannai12survey}.

Cohn--Conway--Elkies--Kumar \cite{Cohn-Conway-Elkies-Kumar07D4} described that 
one can construct a family of designs on $S^{2n-1}$ from a design on $\mathbb{CP}^{n-1}$. 
In this paper, we prove their claim for the case $n = 2$.
Recall that $\mathbb{CP}^{1} \simeq S^2$, and 
therefore the main result of this paper is an algorithm to construct 
spherical designs on $S^3$ from designs on $S^2$.

Let us denote by $\pi : S^3 \rightarrow S^2$ a Hopf map.
Then $(S^3,S^2,\pi)$ is a principal $S^1$-bundle. 
In particular, for each $y \in S^2$, 
the fiber $\pi\inv(y)$ is isomorphic to $S^1$
(see Section \ref{section:results_spherical} for more details).

The following theorem is our main result of this paper:

\begin{thm}[See Theorem \ref{thm:main} for the details]\label{thm:intro:main}
Let $Y$ be a $t$-design on $S^2$ and $\Gamma$ a $2t$-design  $[$resp.~$(2t+1)$-design$]$ on $S^1$.
For each $y \in Y$,
we take a $2t$-design $\Gamma_y$ $[$resp.~$(2t+1)$-design$]$ on $\pi^{-1}(y) \simeq S^1$.
Then the finite subset \[
X := \bigsqcup_{y \in Y} \Gamma_y.
\]
is a $2t$-design $[$resp.~$(2t+1)$-design$]$ on $S^3$ with $|X| = \sum_{y \in Y} |\Gamma_y|$.
\end{thm}

One of important problems of spherical designs 
is to give an algorithm to construct $t$-designs on $S^d$ explicitly.
It should be emphasized that 
Theorem \ref{thm:intro:main} constructs a $2t$-design [resp.~$(2t+1)$-design] on $S^3$ explicitly 
from a given $t$-design $Y$ on $S^2$ 
and the regular $(2t+1)$-gon [resp.~$2(t+1)$-gon] on $S^1$.

Recall that spherical $t$-designs $Y$ on $S^2$ can be constructed as follows:
Kuperberg \cite{Kuperberg05special} showed that 
an interval $t$-design on the open interval $(-1,1)$ with respect to the constant weight can be constructed from the roots of a certain polynomial of degree $\lfloor t/2 \rfloor$.
For an interval $t$-design $\{ \xi_1,\dots, \xi_M \} \subset (-1,1)$ with respect to the constant weight, 
by taking $Y_i$ the regular $(t+1)$-gon on the circle $S^2 \cap \{ (x_1,x_2,x_3) \in \R^3 \mid x_1 = \xi_i \}$ in $S^3$,
we have a $t$-design $Y = \bigcup_{i=1}^M Y_i (\subset S^2)$.
This technique was pointed out by Rebau--Bajnok \cite{Rabau-Bajnok91bounds} and Wagner \cite{Wagner91averaging} (see \cite[Section 2.7]{Bannai-Bannai12survey} for more details).
By combining this and our theorem, we have an algebraic construction of $t$-designs on $S^3$ for each $t$.

We also remark that the idea of Theorem \ref{thm:intro:main} is similar to 
the technique to construct spherical designs from interval designs described above and 
a technique by Ito \cite[Section 8]{Ito04coset_geom} to construct designs on finite group $G$ from designs on a $G$-homogeneous space $\Gamma$.

Let us denote by $N_{S^d}(t)$
the smallest cardinality of a $t$-design on $S^d$.
By Theorem \ref{thm:intro:main}, 
we have the following inequalities:
\begin{align*}
N_{S^3}(2t) \leq (2t+1) N_{S^2}(t) \text{ and } 
N_{S^3}(2t+1) \leq 2(t+1) N_{S^2}(t),
\end{align*}
since regular $(2t+1)$-gon [resp.~$2(t+1)$-gon] on $S^1$ 
is a $2t$-design [resp.~$(2t+1)$-design].
In particular, recall that Chen--Frommer--Lang \cite{Chen-Frommer-Lang11computational} constructed $t$-designs on $S^2$ with $(t+1)^2$ nodes 
for each $t \leq 100$.
Thus, by Theorem \ref{thm:intro:main}, we also obtain 
\begin{align}
N_{S^3} (2t) \leq (2t+1)(t+1)^2 \text{ and }
N_{S^3} (2t+1) \leq 2(t+1)^3 \quad \text{for } t \leq 100. \label{eq:CFL}
\end{align}

By Bondarenko--Radchenko--Viazovska's recent great results \cite{Bondarenko-Radchenko-Viazovska2013OptimalAsymptotic,Bondarenko-Radchenko-Viazovska2014WellSeparated}, the asymptotic bound (conjectured by Korevaar--Meyer \cite{Korevaar-Meyers93} in 1993) for $N_{S^d}(t)$ as 
\begin{equation}
N_{S^d}(t) \ll t^d
\end{equation}
holds for any $d \geq 1$.
Our bounds \eqref{eq:CFL} give a precise estimation of $N_{S^3}(t)$ for $t \leq 100$.

This paper is organized as follows.
In Section \ref{section:results_spherical}, 
we set up notation and state our main theorems.
In Section \ref{section:preliminary}, as a preliminary, 
we give a definition of designs on a general measure space
and show abstract propositions in order to prove our main theorems.
Main results described in Section \ref{section:results_spherical} 
will be proved in Section \ref{section:proof_spherical}
by using propositions in Section \ref{section:preliminary}.

\section{Main results}\label{section:results_spherical}

We fix terminology for spherical designs as follows.

Let us denote by $S^d$ the unit sphere in the $(d+1)$-dimensional Euclidean space $\mathbb{R}^{d+1}$,
and denote by $\mu_{S^d}$ the spherical measure on $S^d$.
We put $|S^d| := \mu_{S^d}(S^d)$.
For each $t \in \mathbb{N}$, we write \[
P_t(\mathbb{R}^{d+1}) := \{\, f \mid \text{$f$ is a polynomial over $\mathbb{C}$ on $\mathbb{R}^{d+1}$ with $\deg f \leq t$} \,\}.
\]
Any element in $P_t(\mathbb{R}^{d+1})$ can be regarded as a $\mathbb{C}$-valued function on $\mathbb{R}^{d+1}$.
We put 
\[
P_t(S^d) := \{\, f|_{S^d} \mid f \in P_t(\mathbb{R}^{d+1}) \,\}.
\]
Then $P_t(S^d)$ is a finite-dimensional functional space on $S^d$.
It is well known that 
\[
\dim_\mathbb{C} P_t(S^d) = \binom{t+d}{d} + \binom{t+d-1}{d-1}.
\]
We define spherical $t$-designs on $S^d$ as follows:

\begin{defn}\label{defn:spherical_design}
A finite subset $X$ of $S^d$ is called a (spherical) $t$-design on $S^d$ 
if  
\[
\frac{1}{|X|} \sum_{x \in X} f(x) = \frac{1}{|S^d|}\int_{S^d} f d\mu_{S^d} \quad \text{for any } f \in P_t(S^d).
\]
\end{defn}

\begin{rem}
In Definition $\ref{defn:spherical_design}$, 
we can replace polynomials over $\mathbb{C}$ to that over $\mathbb{R}$.
In fact, 
the original definition of spherical designs 
in Delsarte--Goethals--Seidel \cite{Delsarte-Goethals-Seidel77spherical}
considered polynomials over $\mathbb{R}$.
In this paper, we discuss over $\mathbb{C}$ 
since monomials on $S^3 \subset \mathbb{C}^2 \simeq \mathbb{R}^4$ over $\mathbb{C}$ can be written easily then that over $\mathbb{R}$.
\end{rem}

Throughout this paper, let us denote by
\begin{align*}
S^3 &:= \left\{\, (a,b) \mid a,b \in \mathbb{C},\ |a|^2 + |b|^2 = 1 \,\right\} \subset \mathbb{C}^2 \simeq \mathbb{R}^4, \\
S^2 &:= \left\{\, (\xi,\eta) \mid \xi \in \mathbb{R},\ \eta \in \mathbb{C},\ \xi^2 + |\eta|^2 = 1 \,\right\} \subset \mathbb{R} \times \mathbb{C} \simeq \mathbb{R}^3, \\
S^1 &:= \left\{ z \in \mathbb{C} \mid |z|=1 \right\} \subset \mathbb{C} \simeq \mathbb{R}^2.
\end{align*}
We fix a Hopf map as follows:
\begin{align*}
\pi : S^3 \rightarrow S^2,\quad (a,b) \mapsto (|a|^2 - |b|^2, 2ab).
\end{align*}
Let us put
\[
(a,b) \cdot z := (az,b\overline{z}) \quad \text{for each } (a,b) \in S^3 \text{ and } z \in S^1.
\]
Then 
\[
S^3 \times S^1 \rightarrow S^3,\ (x,z) \mapsto x \cdot z
\]
defines a right action of $S^1$ on $S^3$
with respect to the usual group structure on $S^1$.
The Hopf map $\pi : S^3 \rightarrow S^2$ is a principal $S^1$-bundle with respect to the right $S^1$-action.
In particular, $S^1$ acts simply-transitively 
on each fiber $\pi\inv(y)$ for $y \in S^2$.
A summary of the Hopf map $S^3 \rightarrow S^2$ can be found in \cite[Part II, \S 20]{Steenrod51}.
 
Here is our main theorem, which will be proved in Section \ref{subsection:proof_main}:
\begin{thm}\label{thm:main}
Let $Y \subset S^2$ be a $t$-design.
For each $y \in Y$, we fix a base point $s_y$ on the fiber $\pi\inv(y)$,
and take a $2t$-design $\Gamma_y \subset S^1$.
Then the finite subset
\[
X(Y,s,\Gamma) := \bigcup_{y \in Y} \{ s_y \cdot \gamma \mid \gamma \in \Gamma_y \}
\]
is a $2t$-design on $S^3$ with $|X(Y,s,\Gamma)| = \sum_{y \in Y} |\Gamma_y|$.
Furthermore, if $\Gamma_y$ is a $(2t+1)$-design for all $y \in Y$, 
then $X(Y,s,\Gamma)$ is a $(2t+1)$-design on $S^3$.
\end{thm}

We note that $X(Y,s,\Gamma)$ depends on the choice of the map
\[
s : Y \rightarrow S^3,\quad y \mapsto s_y.
\]
Therefore, $X(Y,s,\Gamma)$ may be a non-rigid $2t$-design
[resp. $(2t+1)$-design] on $S^3$ 
(see Bannai \cite{Bannai87rigid} for the definition of non-rigid spherical $t$-designs).
In particular, we can not expect that $X(Y,s,\Gamma)$ is a tight $2t$-design on $S^3$.

\begin{example}[Example of Main theorem \ref{thm:main}]
An antipodal subset $Y = \{\, (\pm 1,0) \,\}$ of $S^2 \subset \mathbb{R} \times \mathbb{C}$ is a $1$-design on $S^2$
and a regular $3$-gon \[
\Gamma_3 := \{\, z \in \mathbb{C} \mid |z| = 1,\ z^3 = 1 \,\} \subset S^1 \subset \mathbb{C}
\]
is a $2$-design on $S^1$.
Let us fix a base point $s_y$ of $\pi\inv(y)$ for each $y \in Y$ as follows:
\begin{align*}
s_{(1,0)} := (1,0),\ s_{(-1,0)} := (0,1).
\end{align*}
Then, by Theorem $\ref{thm:main}$, the finite subset 
\[
X := \{\, s_y \cdot z \mid y \in Y, z \in \Gamma_{3} \,\} \subset S^3
\]
is a $2$-design on $S^3$ with $|X| = 6$.
Such $X$ can be written by 
\begin{align*}
X = \{ (1,0), (e^{\sqrt{-1} \frac{2}{3}\pi},0), (e^{\sqrt{-1} \frac{4}{3}\pi},0), 
(0,1), (0,e^{\sqrt{-1} \frac{2}{3}\pi}), (0, e^{\sqrt{-1} \frac{4}{3}\pi}) \}.
\end{align*}
\end{example}

\section{Key ideas for designs on measure spaces}\label{section:preliminary}

In this section, 
we define designs on a general measure space and show some propositions for them.
Main theorem \ref{thm:main} of this paper will be proved by using propositions in this section.

\subsection{Designs on a general measure space}\label{subsection:def_of_design_in_gen}

Let $(\Omega,\mu)$ be a general finite measure space.
We define (weighted) designs for a vector space consisted of $L^1$-integrable functions on $(\Omega,\mu)$ as follows:
\begin{defn}\label{defn:designs}
Let $X$ be a finite subset of $\Omega$ and 
$\lambda : X \rightarrow \mathbb{R}_{>0}$ be a positive weight function on $X$.
For an $L^1$-integrable function $f : \Omega \rightarrow \mathbb{C}$, 
we say that $(X,\lambda)$ is an weighted $f$-design on $(\Omega,\mu)$ if
\[
\sum_{x \in X} \lambda(x) f(x) = \int_{\Omega} f d\mu.
\]
For a vector space $\mathcal{H}$ 
consisted of $L^1$-integrable functions on $\Omega$,
we say that $(X,\lambda)$ is an weighted $\mathcal{H}$-design on $(\Omega,\mu)$
if $(X,\lambda)$ is an weighted $f$-design on $(\Omega,\mu)$ for any $f \in \mathcal{H}$.
Furthermore, 
if $\lambda$ is constant on $X$, 
then $X$ is said to be an $\mathcal{H}$-design on $(\Omega,\mu)$
with respect to the constant $\lambda$.
\end{defn}

\begin{example}
Let $\Omega = S^{d}$, $\mu = (1/|S^d|) \mu_{S^d}$ and $\mathcal{H} = P_t(S^d)$.
Then a finite subset $X$ of $\Omega$ is 
an $\mathcal{H}$-design on $(\Omega,\mu)$ 
with respect to the constant $1/|X|$
if and only if $X$ is a $t$-design on $S^d$. 
\end{example}

Let us consider the cases where 
any constant function on $\Omega$ is in $\mathcal{H}$.
Then for any weighted $\mathcal{H}$-design $(X,\lambda)$ on $(\Omega,\mu)$,
we have $\sum_{x \in X} \lambda(x) = \mu(\Omega)$.
In particular, 
if $X$ is an $\mathcal{H}$-design on $(\Omega,\mu)$ with respect to a positive constant $\lambda$,
then $\lambda = \mu(\Omega)/|X|$.

\begin{rem}
The concept of $\mathcal{H}$-designs on $(\Omega,\mu)$  
is a generalization of that of averaging sets 
on a topological finite measure space $(\Omega,\mu)$ 
$($see \cite{Seymour-Zaslavsky84averaging} for the definition of averaging sets$)$.
In particular, by results of Seymour--Zaslavsky \cite[Main Theorem]{Seymour-Zaslavsky84averaging},
if $(\Omega,\mu)$ is a topological finite measure space and $\Omega$ is path-connected, 
then for any finite-dimensional vector space $\mathcal{H}$ consisted of continuous functions on $\Omega$, 
an $\mathcal{H}$-design on $(\Omega,\mu)$ exists. 
\end{rem}

We give two easy observations for designs on $(\Omega,\mu)$ as follows:

\begin{obs}\label{obs:easy_design}
\begin{itemize}
\item If $\mathcal{H}' \subset \mathcal{H}$,
then any $($weighted$)$ $\mathcal{H}$-design on $(\Omega,\mu)$ 
is also an $($weighted$)$ $\mathcal{H}'$-design on $(\Omega,\mu)$.
\item Let $\lambda$ be a positive constant and $X$, $X'$ are both $\mathcal{H}$-designs on $(\Omega,\mu)$ 
with respect to $\lambda$.
If $X \cap X' = \emptyset$, 
then $X \sqcup X'$ is also an $\mathcal{H}$-design on $(\Omega,\mu)$
with respect to $\lambda$.
\end{itemize}
\end{obs}

\subsection{Key propositions}\label{subsection:keyprop}

Let $(\Omega_1,\mu_1)$, $(\Omega_2,\mu_2)$ be general measure spaces and 
$\pi : \Omega_1 \rightarrow \Omega_2$ a map. 
For each element $\omega \in \Omega_2$, we fix a measure $\mu_{\omega}$ on the fiber $\pi\inv(\omega)$.

Let us take an $L^1$-integrable function $f : \Omega_1 \rightarrow \mathbb{C}$.
We say that the function $f$ satisfies the property $(F)$ if the following holds:
\begin{itemize}
\item For each $\omega \in \Omega_2$, 
the restriction $f|_{\pi\inv(\omega)}$ is also an $L^1$-integrable function on $(\pi\inv(\omega),\mu_\omega)$. 
\item The function 
\[
I_\pi f : \Omega_2 \rightarrow \mathbb{C},\quad \omega \mapsto \int_{\pi\inv(\omega)} f d\mu_{\omega},
\]
is also an $L^1$-integrable function on $\Omega_2$ with
\[
\int_{\Omega_1} f d\mu_1 = \int_{\Omega_2} (I_\pi f) d\mu_2.
\]
\end{itemize}

\begin{rem}
The property $(F)$ for a function $f$ means that 
we can apply ``Fubini's theorem'' for $f$.
\end{rem}

Let us take a finite-dimensional vector space $\mathcal{H}$ 
consisted of $L^1$-integrable functions on $\Omega$ with the property $(F)$.
Then, 
\begin{align*}
I_\pi \mathcal{H} &:= \{\, I_\pi f \mid f \in \mathcal{H} \,\}, \\
\mathcal{H}|_{\pi\inv(\omega)} &:= \{\, f|_{\pi\inv(\omega)} \mid f \in \mathcal{H} \,\} \quad \text{for } \omega \in \Omega_2
\end{align*}
are also finite-dimensional vector spaces consisted of $L^1$-integrable functions.

\begin{example}\label{example:I_pi_for_S3_S2}
Let $(\Omega_1,\mu_1) = (S^{3},(1/|S^3|)\mu_{S^3})$, 
$(\Omega_2,\mu_2) = (S^2,(1/|S^2|)\mu_{S^2})$ 
and $\pi : S^3 \rightarrow S^2$ the Hopf map.
For each $y \in S^2$, we put the $S^1$-invariant probability measure $\mu_{y}$ on the fiber $\pi\inv(y)$. 
In Section $\ref{subsection:proof_main}$, 
we will prove that any $L^1$-integrable function on $\Omega_1 = S^3$ satisfies the property $(F)$,
and  
\begin{align*}
I_\pi (P_{t}(S^3)) = P_{\lfloor \frac{t}{2} \rfloor}(S^2) 
\end{align*} 
for each $t$
$($see Lemma $\ref{lem:Fubini_for_Hopf}$ and Lemma $\ref{lem:func_sp_sphere}$ for more details$)$. 
\end{example}

Let $Y$ be a finite subset of $\Omega_2$ and $\lambda_Y$ a positive function on $Y$.
For each $y \in Y$, 
we take a finite subset $\Gamma_y$ of $\pi\inv(y)$ and 
a positive function $\lambda_{\Gamma_y}$ on $\Gamma_y$.
We denote by 
\begin{equation}
X(Y,\Gamma) := \bigsqcup_{y \in Y} \Gamma_y \label{eq:lift_gen}
\end{equation}
and define a positive function on $X(Y,\Gamma)$ by 
\[
\lambda_X : X(Y,\Gamma) = \bigsqcup_{y \in Y} \Gamma_y \rightarrow \mathbb{R}_{>0},\quad 
x \mapsto \lambda_Y(y) \cdot \lambda_{\Gamma_y}(x) \quad \text{if } x \in \Gamma_y.
\]

Then the next proposition holds:

\begin{prop}\label{prop:general_sp_lift}
Let $(Y,\lambda_Y)$ be an weighted $(I_\pi \mathcal{H})$-design on $(\Omega_2,\mu_2)$
and $(\Gamma_y,\lambda_{\Gamma_y})$ an weighted $\mathcal{H}|_{\pi\inv(y)}$-design on $(\pi\inv(y),\mu_y)$ 
for each $y \in Y$.
Then $(X(Y,\Gamma),\lambda_X)$ defined above 
is an weighted $\mathcal{H}$-design on $(\Omega_1,\mu_1)$.
\end{prop}

The proof of Proposition \ref{prop:general_sp_lift} is given in the next subsection.

The next corollary, which will be used in the proof of Theorem \ref{thm:main} (see Section \ref{subsection:proof_main}), follows from Proposition \ref{prop:general_sp_lift} immediately.

\begin{cor}\label{cor:gen_lift_const}
In the setting of Proposition $\ref{prop:general_sp_lift}$,
suppose that $Y$ is a $(I_\pi \mathcal{H})$-design on $(\Omega_2,\mu_2)$ 
with respect to a positive constant $\lambda_{Y}$, 
and there exists a positive constant $\lambda_{\Gamma}$ such that 
for any $y \in Y$, 
the set $\Gamma_y$ is an $\mathcal{H}|_{\pi\inv(y)}$-design on $(\pi\inv(y),\mu_y)$ 
with respect to $\lambda_{\Gamma}$.
Then $X(Y,\Gamma)$ is an $\mathcal{H}$-design on $(\Omega_1,\mu_1)$ with respect to the constant $\lambda_Y \cdot \lambda_{\Gamma}$.
\end{cor}

\subsection{Proofs of key propositions}

By the definition of weighted designs,
the proof of Proposition \ref{prop:general_sp_lift} 
is reduced to the showing the next lemma:

\begin{lem}\label{lem:general1}
Let $f$ be a $L^1$-integrable function on $\Omega$ with the property $(F)$.
Suppose that $(Y,\lambda_Y)$ is an weighted $(I_\pi f)$-design on $(\Omega_2,\mu_2)$ 
and $(\Gamma_y,\lambda_{\Gamma_y})$ is an weighted $(f|_{\pi\inv(y)})$-design on $(\pi\inv(y),\mu_y)$ 
for each $y \in Y$.
Then $(X(Y,\Gamma),\lambda_X)$ is an weighted $f$-design on $(\Omega_1,\mu_1)$ 
$($see \eqref{eq:lift_gen} for the notation of $X(Y,\Gamma)$$)$.
\end{lem}

Lemma \ref{lem:general1} claims that 
if we have weighted designs on $\Omega_2$ and that on some fibers,
then we have an weighted design on $\Omega_1$.

\begin{proof}[Proof of Lemma $\ref{lem:general1}$]
Since $(\Gamma_y,\lambda_{\Gamma_y})$ is an weighted $(f|_{\pi\inv(y)})$-design on $(\pi\inv(y),\mu_y)$ for each $y \in Y$, we have 
\begin{align*}
\sum_{x \in X(Y,\Gamma)} \lambda_X(x) f(x)
	&= \sum_{y \in Y} \sum_{\gamma_y \in \Gamma_y} \lambda_Y(y) \lambda_{\Gamma_y}(\gamma_y) f(\gamma_y) \\
	&= \sum_{y \in Y} \lambda_Y(y) \left( \sum_{\gamma_y \in \Gamma_y} \lambda_{\Gamma_y}(\gamma_y) f(\gamma_y) \right) \\
	&= \sum_{y \in Y} \lambda_Y(y) \int_{\pi\inv(y)} f d\mu_y.
\end{align*}
Furthermore, 
since $(Y,\lambda_Y)$ is an weighted $(I_\pi f)$-design on $(\Omega_2,\mu_2)$,
we have 
\begin{align*}
\sum_{y \in Y} \lambda_Y(y) \int_{\pi\inv(y)} f d\mu_y 
	&= \sum_{y \in Y} \lambda_Y(y) (I_\pi f)(y) \\
	&= \int_{\Omega_2} (I_\pi f) d\mu_2 \\
	&= \int_{\Omega_1} f d\mu_1. 
\end{align*}
This completes the proof.
\end{proof}

\section{Proof of Main result}\label{section:proof_spherical}

In this section,
we prove Theorem \ref{thm:main} by using the results in Section \ref{subsection:keyprop}.

\subsection{Local trivializations of the Hopf map}\label{subsection:trivialization}

In this subsection, we recall local trivializations of 
the Hopf map $\pi : S^3 \rightarrow S^2$ defined in Section \ref{section:results_spherical}.

Let us take an open covering $\{ U_+, U_- \}$ of $S^2 \subset \mathbb{R} \times \mathbb{C}$ as
\begin{align*}
U_+ = \{\, (\xi,\eta) \in S^2 \mid \xi \neq -1 \,\}, \ 
U_- = \{\, (\xi,\eta) \in S^2 \mid \xi \neq 1 \,\}.
\end{align*}
Then we have local trivializations of the $S^1$-bundle $\pi : S^3 \rightarrow S^2$ as
\begin{align*}
U_+ \times S^1 \simrightarrow \pi\inv(U_+),\quad
	((\xi,\eta),z) \mapsto \left( \sqrt{\frac{1+\xi}{2}} z, \sqrt{\frac{1}{2(1+\xi)}} \eta \overline{z} \right), \\
U_- \times S^1 \simrightarrow \pi\inv(U_-),\quad 
	((\xi,\eta),z) \mapsto \left( \sqrt{\frac{1}{2(1-\xi)}} \eta z, \sqrt{\frac{1-\xi}{2}} \overline{z} \right).
\end{align*}

In particular, for each $y = (\xi,\eta) \in U_+$, the fiber $\pi\inv(y)$ can be written by
\begin{align}
\pi\inv(y) = \left\{\, \left( \sqrt{\frac{1+\xi}{2}} z, \sqrt{\frac{1}{2(1+\xi)}} \eta \overline{z} \right) 
\mid z \in S^1 \,\right\} \subset S^3. \label{eq:fiber_+}
\end{align}
Similarly, for each $y \in (\xi,\eta) \in U_-$, we have
\begin{align}
\pi\inv(y) = \left\{\, \left( \sqrt{\frac{1}{2(1-\xi)}} \eta z, \sqrt{\frac{1-\xi}{2}} \overline{z} \right) \mid z \in S^1 \,\right\} \subset S^3. \label{eq:fiber_-}
\end{align}

\begin{rem}
In Theorem $\ref{thm:main}$,
we need to take a base point $s_y$ on $\pi\inv(y)$ for a given $y \in S^2$.
By using the explicit form of $\pi\inv(y)$ above, 
one can choose $s_y$ explicitly.
\end{rem}

\subsection{Proof of Theorem $\ref{thm:main}$}\label{subsection:proof_main}

Throughout this subsection, 
we denote by $\mu'_{S^d} := (1/|S^d|) \mu_{S^d}$.
Then $\mu'_{S^d}$ is the $O(d+1)$-invariant Haar measure on $S^d$ with $\mu'_{S^d}(S^d) = 1$.

Let $\pi : S^3 \rightarrow S^2$ be the Hopf map defined in Section \ref{section:results_spherical}.
For simplicity, we fix a base point $s_y$ on a fiber $\pi\inv(y)$ for each $y \in S^2$.
Note that we do not assume that the map $s : S^2 \rightarrow S^3$ with $s \circ \pi = \id_{S^2}$ is continuous 
(in fact, such a continuous map does not exist).
Then we have an isomorphism \[
\iota_y : S^1 \rightarrow \pi\inv(y),\quad z \mapsto s_y \cdot z.
\]
For each $y \in S^2$, 
we consider the induced measure $\mu'_y$ on $\pi\inv(y)$ 
by the normalized measure $\mu'_{S^1}$ on $S^1$.
Such the probability measure $\mu'_{y}$ on $\pi\inv(y)$ 
does not depend on the choice of 
the base point $s_y$ since $\mu'_{S^1}$ is invariant by the $S^1$-action.

To prove Theorem \ref{thm:main}, we show the next two lemmas.

\begin{lem}\label{lem:Fubini_for_Hopf}
Any $L^1$-integrable function on $S^3$ satisfies the property $(F)$ 
with respect to 
the Hopf map $\pi : S^3 \rightarrow S^2$, 
the normalized spherical measures $\mu'_{S^3}$, $\mu'_{S^2}$ and the measure $\mu'_y$ 
on $\pi\inv(y)$ for each $y \in S^2$ defined above
$($see Section $\ref{subsection:keyprop}$ for the definition of the property $(F)$$)$.
\end{lem}

\begin{lem}\label{lem:func_sp_sphere}
For any $t \in \mathbb{N}$, we have  
\begin{align*}
\iota_y^* (P_{t}(S^3)|_{\pi\inv(y)}) = P_{t}(S^1) \quad \text{for any } y \in S^2, \\
I_\pi (P_{t}(S^3)) = P_{\lfloor \frac{t}{2} \rfloor}(S^2) \text{ and } \pi^* (P_{\lfloor \frac{t}{2} \rfloor}(S^2)) \subset P_{t}(S^3)
\end{align*} 
$($see Section $\ref{subsection:keyprop}$ for the definition of $I_\pi$$)$.
\end{lem}

One can observe that 
Theorem \ref{thm:main} follows from Corollary \ref{cor:gen_lift_const}, Lemma \ref{lem:Fubini_for_Hopf} and Lemma \ref{lem:func_sp_sphere}.

\begin{proof}[Proof of Lemma $\ref{lem:Fubini_for_Hopf}$]
Let us denote by 
\begin{align*}
S^3 &= \{\, ((\cos \varphi) e^{\sqrt{-1}\theta_1} ,(\sin \varphi) e^{\sqrt{-1} \theta_2}) \mid 0 \leq \varphi \leq \frac{\pi}{2}, 0 \leq \theta_1,\theta_2 < 2\pi \,\} \subset \mathbb{C}^2, \\
S^2 &= \{\, (\cos \psi, (\sin \psi) e^{\sqrt{-1} \phi}) \mid 0 \leq \psi \leq \pi,\ 0 \leq \phi < 2\pi  \,\} \subset \mathbb{R} \times \mathbb{C}, \\
S^1 &= \{\, e^{\sqrt{-1} \theta} \mid 0 \leq \theta < 2\pi \,\} \subset \mathbb{C}.
\end{align*}
Then the volume forms corresponding to the normalized measures 
$\mu'_{S^d}$ ($d = 1,2,3$) can be written by
\begin{align*}
d\mu'_{S^3} &= \frac{1}{4\pi^2} (\sin 2\varphi) d\varphi d\theta_1 d\theta_2 \quad (0 \leq \varphi \leq \frac{\pi}{2}, 0 \leq \theta_1,\theta_2 < 2\pi), \\
d\mu'_{S^2} &= \frac{1}{4\pi} (\sin \psi) d\psi d\phi \quad (0 \leq \psi \leq \pi,\ 0 \leq \phi < 2\pi), \\
d\mu'_{S^1} &= \frac{1}{2\pi} d\theta \quad (0 \leq \theta < 2\pi).
\end{align*}
We put 
\begin{align*}
U_+	&= \{\, (\cos \psi, (\sin \psi) e^{\sqrt{-1} \phi}) \mid 0 \leq \psi < \pi,\ 0 \leq \phi < 2\pi \,\} \subset S^2, \\
\pi\inv(U_+) &= 
\{\, ((\cos \varphi) e^{\sqrt{-1}\theta_1},(\sin \varphi) e^{\sqrt{-1} \theta_2}) \mid 0 \leq \varphi < \frac{\pi}{2}, 0 \leq \theta_1,\theta_2 < 2\pi \,\} \subset S^3.
\end{align*}
Then the isomorphism between $U_+ \times S^1$ and $\pi\inv(U_+)$ given in Section \ref{subsection:trivialization} can be written by 
\begin{align*}
U_+ \times S^1 &\rightarrow \pi\inv(U_+),\\
(\cos \psi, (\sin \psi) e^{\sqrt{-1}\phi}, e^{\sqrt{-1}\theta}) &\mapsto 
((\cos \frac{\psi}{2}) e^{\sqrt{-1} \theta}, (\sin \frac{\psi}{2}) e^{\sqrt{-1}(\phi - \theta)}).
\end{align*}
Under this isomorphism, we have 
\begin{align*}
\varphi = \frac{\psi}{2}, \ \theta_1 = \theta,\ \theta_2 = \phi - \theta.
\end{align*}
Thus, 
\begin{align*}
d\mu'_{\pi\inv(U_+)} &= \frac{1}{4\pi^2} (\sin 2\varphi) d\varphi d\theta_1 d\theta_2 \\
	&= \frac{1}{8\pi^2} (\sin \psi) d\psi d\phi d\theta \\
	&= d\mu'_{U_+} d\mu'_{S^1},
\end{align*}
where we put $d\mu'_{\pi\inv(U_+)} := d\mu'_{S^3}|_{\pi\inv(U_+)}$ and $d\mu'_{U_+} := d\mu'_{S^2}|_{U_+}$.
Therefore, we can apply Fubini's theorem for 
\[
\pi\inv(U_+) \simeq U_+ \times S^1.
\]
One can observe that $\mu'_{S^2}(S^2 \setminus U_+) = 0$ and 
$\mu'_{S^3}(S^3 \setminus \pi\inv(U_+)) = 0$.
In particular, for any $L^1$-integrable function $f$ on $S^3$,
we have 
\begin{align*}
\int_{S^3} f d\mu'_{S^3} 
	&= \int_{\pi\inv(U_+)} f d\mu'_{\pi\inv(U_+)} \\
	&= \int_{(\xi,\eta) \in U_+} \int_{z \in S^1} f(\xi,\eta,z) d\mu'_{S^1}(z) d\mu'_{U_+}(\xi,\eta) \\
	&= \int_{U_+} (I_\pi f) d\mu'_{U_+} \\
	&= \int_{S^2} (I_\pi f) d\mu'_{S^2}.
\end{align*}
This completes the proof.
\end{proof}

\begin{proof}[Proof of Lemma $\ref{lem:func_sp_sphere}$]
First, we shall prove that 
\begin{align}
\iota_y^*(P_{t}(S^3)|_{\pi\inv(y)}) &=  P_{t}(S^1) \quad \text{for any } y \in S^2, \label{align:prove_fiber} \\
I_\pi (P_{t}(S^3)) &\subset P_{\lfloor \frac{t}{2} \rfloor}(S^2) \label{align:prove_Ipi}.
\end{align}
Let us fix any $n \leq t$ and denote by $f_{i,j,k,l}(a,\overline{a},b,\overline{b}) := a^i \overline{a}^{j} b^k \overline{b}^l$ the monomial on $\mathbb{R}^4 \simeq \mathbb{C}^2$ of degree $n = i+j+k+l$.
We also denote by the same letter $f_{i,j,k,l}$ the restricted function on $S^3$ 
of the monomial $f_{i,j,k,l}(a,\overline{a},b,\overline{b})$.
By Lemma \ref{lem:Fubini_for_Hopf},
the function $f_{i,j,k,l}$ on $S^3$ satisfies the property $(F)$.
To prove \eqref{align:prove_fiber} and \eqref{align:prove_Ipi}, 
it suffices to show that:
\begin{itemize}
\item\label{item:fijkl_S1} 
For each $y \in S^2$, 
the function $\iota_{y}^*(f_{i,j,k,l}|_{\pi\inv(y)})$ 
is a monomial on $S^1$ of degree $|i-j-k+l|$,
\item $I_\pi f_{i,j,k,l} \in P_{\lfloor \frac{n}{2} \rfloor}(S^2)$.
\end{itemize}
For each $y = (\xi,\eta) \in S^2$,
by the explicit formula \eqref{eq:fiber_+} and \eqref{eq:fiber_-} of the fiber $\pi\inv(y)$ 
given in Section \ref{subsection:trivialization}, 
there exists a constant $c_{i,j,k,l}(y) \in \mathbb{C}$ such that 
\begin{align*}
f_{i,j,k,l} (\iota_y(z)) 
	&= c_{i,j,k,l}(y) z^{i-j-k+l} \\
	&= \begin{cases} 
		c_{i,j,k,l}(y) z^{i-j-k+l} \quad \text{if $i-j-k+l \geq 0$,}\\
		c_{i,j,k,l}(y) \overline{z}^{-i+j+k-l} \quad \text{if $i-j-k+l < 0$.} \\
	   \end{cases}
\end{align*}
Thus, $\iota_y^*(f_{i,j,k,l}|_{\pi\inv(y)})$ is a monomial on $S^1$ of degree $|i-j-k+l|$.
Furthermore, for each $y \in S^2$, we have 
\begin{align*}
(I_\pi f_{i,j,k,l}) (y) 
	&= c_{i,j,k,l}(y) \int_{S^1} z^{i-j-k+l} d\mu'_{S^1} \\
	&= \begin{cases} c_{i,j,k,l}(y) \quad \text{if } i+l = j+k, \\ 0 \quad \text{otherwise}. \end{cases}
\end{align*}
In particular, if $i+l \neq j+k$, then  
\begin{align*}
(I_\pi f_{i,j,k,l}) (y) = 0 \quad \text{for any } y \in S^2.
\end{align*}
Therefore, let us consider the cases where $n = i+j+k+l = 2m$ is even and $i+l = j+k = m$.
For each $y = (\xi,\eta) \in S^2$, 
by \eqref{eq:fiber_+}, \eqref{eq:fiber_-} and $|\eta|^2 = (1-\xi)(1+\xi)$,
we have 
\begin{align*}
I_\pi f_{i,j,k,l} (y) &= c_{i,j,k,l}(y) \\
	&= 
\begin{cases} 
\frac{1}{\sqrt{2}^n} (\sqrt{1+\xi})^{i+j-k-l} \eta^{k} \overline{\eta}^{l} \quad \text{if } \xi \neq -1, \\
\frac{1}{\sqrt{2}^n} (\sqrt{1-\xi})^{-i-j+k+l} \eta^{i} \overline{\eta}^{j} \quad \text{if } \xi \neq 1, 
\end{cases} \\
	&= 
\begin{cases} 
\frac{1}{2^m} (1+\xi)^{i-k} \eta^{k} \overline{\eta}^{m-i} \quad \text{if } \xi \neq -1, \\
\frac{1}{2^m} (1-\xi)^{-i+k} \eta^{i} \overline{\eta}^{m-k} \quad \text{if } \xi \neq 1, 
\end{cases} \\
	&= 
\begin{cases} 
\frac{1}{2^m} (1+\xi)^{i-k} \eta^{k} \overline{\eta}^{m-i} \quad \text{if } \xi \neq -1, i \geq k, \\
\frac{1}{2^m} (1+\xi)^{i-k} |\eta|^{2(k-i)} \eta^{i} \overline{\eta}^{m-k} \quad \text{if } \xi \neq -1, i \leq k \\
\frac{1}{2^m} (1-\xi)^{-i+k} |\eta|^{2(i-k)} \eta^{k} \overline{\eta}^{m-i} \quad \text{if } \xi \neq 1, i \geq k, \\
\frac{1}{2^m} (1-\xi)^{-i+k} \eta^{i} \overline{\eta}^{m-k} \quad \text{if } \xi \neq 1, i \leq k, 
\end{cases} \\
	&= 
\begin{cases} 
\frac{1}{2^m} (1+\xi)^{i-k} \eta^{k} \overline{\eta}^{m-i} \quad \text{if } i \geq k, \\
\frac{1}{2^m} (1-\xi)^{k-i} \eta^{i} \overline{\eta}^{m-k} \quad \text{if } i \leq k.
\end{cases}
\end{align*}
Hence, we have $I_\pi f_{i,j,k,l} \in P_{m}(S^2) = P_{n/2}(S^2)$.

Since $\mu'_{y}(\pi\inv(y)) = 1$, 
we have that $I_\pi \circ \pi^*$ is identity on $P_{\lfloor t/2 \rfloor}(S^2)$.
Therefore, to complete the proof of our claim, 
we only need to show that $\pi^* P_{\lfloor t/2 \rfloor}(S^2) \subset P_t(S^3)$.
Let us take a monomial
$h_{i,j,k}(\xi,\eta,\overline{\eta}) := \xi^i \eta^j \overline{\eta}^l$ 
on $\mathbb{R}^3 \simeq \mathbb{R} \times \mathbb{C}$ 
of degree $i+j+k = n \leq \lfloor t/2 \rfloor$.
We also denote by the same letter $h_{i,j,k}$ 
the restricted function on $S^2$ of the monomial $h_{i,j,k}(\xi,\eta,\overline{\eta})$.
Our goal is to show that 
\[
\pi^* h_{i,j,k} \in P_{2n}(S^3).
\]
The function $\pi^* h_{i,j,k}$ on $S^3$ can be written by 
\begin{align*}
(\pi^* h_{i,j,k})(a,b) 
	&= h_{i,j,k}(\pi(a,b)) \\
	&= (|a|^2-|b|^2)^i (2ab)^j (\overline{2ab})^k \\
	&= 2^{j+k} (a\overline{a}-b\overline{b})^i a^j\overline{a}^kb^j\overline{b}^k.
\end{align*}
Hence, we have $\pi^* h_{i,j,k} \in P_{2n}(S^3)$.
This completes the proof.
\end{proof}

\section*{Concluding remarks}
It is well known that $S^3$ admits a compact Lie group structure $($such a compact Lie group is called $SU(2)$$)$ and 
for a maximal torus $S^1$ of $S^3$, 
the Hopf map $\pi : S^3 \rightarrow S^2$ can be regarded as a quotient map 
from the Lie group $S^3$ to the quotient space $S^3/S^1 \simeq S^2$.
In a future paper,
we will discuss a generalization of the results in this paper
to a relation among designs on $G/K$, that on $G/{K'}$ and that on $K'/K$
for a compact Lie group $G$ and closed subgroups $K$ and $K'$ of $G$ with $K \subset K'$.
In particular, 
by considering $G = SU(n)$, $K = SU(n-1)$ and $K' = S(U(1) \times U(n-1))$,
we will obtain an algorithm to construct a family of spherical designs on $S^{2n-1}$ from a design on $\mathbb{CP}^{n-1}$.

Furthermore, constructions of ``extremal spherical designs'' and ``well conditioned spherical designs'', which are spherical designs on $S^2$ with some nice properties from the viewpoint of numerical analysis, were studied by \cite{An-Chen-Sloan-Womersley2010well-conditioned, Chen-Frommer-Lang11computational}. 
By Theorem \ref{thm:main}, 
if we have an extremal [resp.~well conditioned] $t$-design $Y$ on $S^2$,
then we obtain a $2t$-design $X$ on $S^3$ 
as a ``product'' of $Y$ and a regular $(2t+1)$-gon on $S^1$.
Since $Y$ has a nice property as a design on $S^2$,
we may expect that $X$ also has nice properties as a design on $S^3$.
What are such nice properties for designs on $S^3$?
This is also a future work.

\section*{Acknowledgements.}
The author would like to give heartfelt thanks to 
Eiichi Bannai, Tatsuro Ito and Toshiyuki Kobayashi 
whose suggestions were of inestimable value for this paper.
The author would also like to thank to Hirotake Kurihara and Masanori Sawa
whose comments made enormous contribution to this paper.
Finally, the author is also indebted to Congpei An 
whose comments improved this paper.

\providecommand{\bysame}{\leavevmode\hbox to3em{\hrulefill}\thinspace}
\providecommand{\MR}{\relax\ifhmode\unskip\space\fi MR }
\providecommand{\MRhref}[2]{%
  \href{http://www.ams.org/mathscinet-getitem?mr=#1}{#2}
}
\providecommand{\href}[2]{#2}


\begin{thebibliography}{10}

\bibitem{An-Chen-Sloan-Womersley2010well-conditioned}
Congpei An, Xiaojun Chen, Ian~H. Sloan, and Robert~S. Womersley,
  \emph{Well Conditioned
  Spherical Designs for Integration and Interpolation on the Two-Sphere}, SIAM
  J. Numer. Anal. \textbf{48} (2010), \href{http://epubs.siam.org/doi/abs/10.1137/100795140}{2135--2157}.

\bibitem{Bannai87rigid}
Eiichi Bannai, \emph{Rigid spherical $ t
  $-designs and a theorem of {Y}. {H}ong}, J. Fac. Sci. Univ. Tokyo Sect. IA
  Math. \textbf{34} (1987), \href{http://hdl.handle.net/2261/1675}{485--489}.

\bibitem{Bannai-Bannai12survey}
Eiichi Bannai and Etsuko Bannai,
  \emph{A
  survey on spherical designs and algebraic combinatorics on spheres},
  European J. Combin. \textbf{30} (2009), \href{http://www.sciencedirect.com/science/article/pii/S0195669808002400}{1392--1425}.

\bibitem{Bannai-Damerell79tight}
Eiichi Bannai and R.~M.~Damerell,
  \emph{Tight spherical designs. {I}, {II}},
  J. Math. Soc. Japan \textbf{31} (1979), \href{http://projecteuclid.org/euclid.jmsj/1240319488}{199--207} bid 
 J. London Math. Soc. \textbf{21} (1980), 
\href{http://jlms.oxfordjournals.org/content/s2-21/1/13}{13--30}.


\bibitem{Bondarenko-Radchenko-Viazovska2013OptimalAsymptotic}
Andriy Bondarenko, Danylo Radchenko, and Maryna Viazovska, \emph{Optimal
  asymptotic bounds for spherical designs}, Ann.~of Math. \textbf{178} (2013),
  \href{http://annals.math.princeton.edu/2013/178-2/p02}{443--452}.

\bibitem{Bondarenko-Radchenko-Viazovska2014WellSeparated}
Andriy Bondarenko, Danylo Radchenko, and Maryna Viazovska, \emph{Well-Separated Spherical Designs}, Constr. Approx. \textbf{41} (2015),
  \href{http://link.springer.com/article/10.1007%2Fs00365-014-9238-2}{93--112}.

\bibitem{Chen-Frommer-Lang11computational}
Xiaojun Chen, Andreas Frommer, and Bruno Lang,
  \emph{Computational
  existence proofs for spherical {$t$}-designs}, Numer. Math. \textbf{117}
  (2011), \href{http://www.springerlink.com/content/41674453126n618u/}{289--305}.

\bibitem{Cohn-Conway-Elkies-Kumar07D4}
Henry Cohn, John~H. Conway, Noam~D. Elkies, and Abhinav Kumar,
  \emph{The
  {$D_4$} root system is not universally optimal}, Experiment. Math.
  \textbf{16} (2007), \href{http://projecteuclid.org/getRecord?id=euclid.em/1204928532}{313--320}.

\bibitem{Delsarte-Goethals-Seidel77spherical}
P.~Delsarte, J.~M. Goethals, and J.~J. Seidel, \emph{Spherical codes and
  designs}, Geometriae Dedicata \textbf{6} (1977), 363--388.

\bibitem{Ito04coset_geom}
Tatsuro Ito,
  \emph{Designs in a
  coset geometry: {D}elsarte theory revisited}, European J. Combin.
  \textbf{25} (2004), \href{http://dx.doi.org/10.1016/S0195-6698(03)00102-1}{229--238}.

\bibitem{Kuperberg05special}
Greg Kuperberg,
  \emph{Special
  moments}, Adv. in Appl. Math. \textbf{34} (2005), \href{http://www.sciencedirect.com/science/article/pii/S0196885804001423}{853--870}.

\bibitem{Korevaar-Meyers93}
Korevaar, J. and Meyers, J. L. H.
  \emph{Spherical {F}araday cage for the case of equal point charges
              and {C}hebyshev-type quadrature on the sphere}, 
Integral Transform. Spec. Funct.  \textbf{1} (1993), 
\href{http://dx.doi.org/10.1080/10652469308819013}{105--117}.

\bibitem{Rabau-Bajnok91bounds}
Patrick Rabau and Bela Bajnok,
  \emph{Bounds for the
  number of nodes in {C}hebyshev type quadrature formulas}, J. Approx. Theory
  \textbf{67} (1991), \href{http://dx.doi.org/10.1016/0021-9045(91)90018-6}{199--214}.


\bibitem{Seymour-Zaslavsky84averaging}
P.~D. Seymour and Thomas Zaslavsky,
  \emph{Averaging sets: a
  generalization of mean values and spherical designs}, Adv. in Math.
  \textbf{52} (1984), \href{http://dx.doi.org/10.1016/0001-8708(84)90022-7}{213--240}.

\bibitem{Steenrod51}
Norman Steenrod, \emph{The {T}opology of {F}ibre {B}undles}, Princeton
  Mathematical Series, vol. 14, Princeton University Press, Princeton, N. J.,
  1951.

\bibitem{Wagner91averaging}
Gerold Wagner, \emph{On averaging
  sets}, Monatsh. Math. \textbf{111} (1991), \href{http://dx.doi.org/10.1007/BF01299278}{69--78}.

\end{thebibliography}
\end{document}